\documentclass[12pt]{article}
\usepackage[T1]{fontenc}
\usepackage{amsmath, amsthm, amssymb}
\usepackage[ansinew]{inputenc}
\usepackage{amsfonts}
\usepackage{amssymb}
\usepackage{xcolor}
\usepackage{enumitem}% http://ctan.org/pkg/enumitem
\usepackage{amsthm}% http://ctan.org/pkg/amsthm
\usepackage{dsfont}
\usepackage{tikz-cd}
\usetikzlibrary{arrows.meta, positioning}
\DeclareMathAlphabet{\mathbbmsl}{U}{bbm}{m}{sl}

\def\min{\mathop{\rm min}}

\newtheorem{theorem}{Theorem}[section]

\newtheorem{corollary}[theorem]{Corollary}
\newtheorem{example}[theorem]{Example}

\newtheorem{remark}[theorem]{Remark}
\newtheorem{lemma}[theorem]{Lemma}
\newtheorem{definition}[theorem]{Definition}
\newtheorem{proposition}[theorem]{Proposition}
\usepackage[numbers,sort&compress]{natbib}
\usepackage{float}
\usepackage[colorlinks=true,
            linkcolor=blue,
            citecolor=blue,
            urlcolor=blue]{hyperref}
\begin{document}
\title{\large \textbf{On the Convexity of the Solution Set of  Linear Complementarity Problem over Tensor Spaces}}

\author{Sonali Sharma~\!\href{https://orcid.org/0000-0003-0862-9711}{%
\raisebox{-0.2ex}{\includegraphics[width=0.75em]{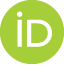}}}$^{a,1}$, V. Vetrivel$^{a,2},$ Jein-Shan Chen$^{b,3}$ \\
{\small$^{a}$Department of Mathematics, Indian Institute of Technology Madras, Chennai, India}\\
{\small$^{b}$Department of Mathematics, National Taiwan Normal University, Taipei, Taiwan}\\
{\small $^{1}$Email: ma24r005@smail.iitm.ac.in, ssonali836@gmail.com}\\
{\small $^{2}$Email: vetri@iitm.ac.in}\\
{\small $^{3}$Email: jschen@math.ntnu.edu.tw}\\
}
\date{}

\maketitle

\begin{abstract}
This paper investigates the convexity of the solution set of the linear complementarity problems over tensor spaces (TLCPs). We introduce the notion of a $T$-column sufficient tensor and study its properties and relationships with several structured tensors. An equivalent condition for the convexity of the solution set of the $\mathrm{TLCP}$ is established. In addition, sufficient conditions for uniqueness and for feasibility implying solvability are derived.
\end{abstract}
\textbf{{Keywords:}} Linear complementarity problem, tensor space, convexity, $T$-column sufficient, uniqueness.\\
\textbf{Mathematics subject classification:} 90C20, 90C26, 90C33. 

\section{Introduction}
Given a matrix $\mathbf{M} \in \mathbb{R}^{n \times n}$ and a vector $\mathbf{q} \in \mathbb{R}^{n}$, the classical linear complementarity problem (LCP), denoted by $\mathrm{LCP}(\mathbf{M},\mathbf{q})$ is to find a vector $\mathbf{z} \in \mathbb{R}^{n}$ such that
\begin{equation*}
 \mathbf{z} \geq \mathbf{0},~ \mathbf{M}\mathbf{z}+\mathbf{q} \geq \mathbf{0},~\text{and}~\mathbf{z}^{T}(\mathbf{M}\mathbf{z}+\mathbf{q}) =0,   
\end{equation*}
where the inequality is understood component-wise. The linear complementarity problem has been extensively studied due to its important applications in optimization, engineering, and economic equilibrium models. Many properties of the solution set of $\mathrm{LCP}$ depend strongly on the structural properties of the matrix (see \cite{MR3396730}).

Recently, Song and Qi \cite{MR3341670} introduced the tensor complementarity problem (TCP), which extends the classical complementarity framework to higher-order tensors. For a tensor $\mathcal{M} \in \mathbb{R}^{[m,n]}$  and a vector $\mathbf{q} \in \mathbb{R}^{n}$, the TCP consists of finding $\mathbf{z} \in \mathbb{R}^{n}$ such that 
\begin{equation*}
  \mathbf{z} \geq \mathbf{0},~\mathcal{M}\mathbf{z}^{m-1}+\mathbf{q} \geq \mathbf{0},~\text{and}~\mathbf{z}^{T}(\mathcal{M}\mathbf{z}^{m-1}+\mathbf{q})=0,  
\end{equation*}
where $\mathcal{M}\mathbf{z}^{m-1}$  is a homogeneous polynomial of degree $m-1$ and its $i$-th component is given as
$$(\mathcal{M}{\bf z}^{m-1})_{i} = \displaystyle{\sum_{i_{2},...,i_{m} =1}^n}{m_{ii_{2}...i_{m}}z_{i_{2}}...z_{i_{m}}},$$
with $\mathcal{M} = (m_{i_{1}i_{2}...i_{m}}),$ for all $i_{j} \in [n]~(j \in [m]).$ When $m=2$, the TCP reduces to the  linear complementarity problem. 

Since the introduction of the TCP, it has been widely investigated, and several classes of structured tensors have been introduced to study the existence, finiteness, boundedness, and uniqueness of solutions. For instance, nondegenerate tensors are related to the finiteness of the solution set \cite{MR4310678},  (strictly) semi-positive and column sufficient tensors are closely associated with uniqueness properties \cite{MR3501398, MR3778365}. Furthermore, $P$-tensors and strong $P$-tensors \cite{MR3513266} play an important role in guaranteeing the compactness and uniqueness of solutions. For a brief overview of the theory, solution methods, and applications of the TCP, we refer the reader to \cite{MR3989294,MR4023437,MR3998357} and references therein. Many generalizations of the TCP and the properties of their solution set with applications have also been investigated in the literature (see \cite{LI2026116873,li2025extended,sharma2025extended,shang2023structured,li2024strict} and references therein).

More recently, Li and Huang \cite{Li2020} introduced the linear complementarity problem over tensor spaces (TLCP) by using the contraction product of tensors. Let $\mathcal{M}$ be in  $\mathbb{R}^{[2m,n]}$ and $\mathcal{Q} \in \mathbb{R}^{[m,n]}$. The TLCP, denoted by $\mathrm{TLCP}(\mathcal{M},\mathcal{Q})$,  is to find a tensor $\mathcal{Z} \in \mathbb{R}^{[m,n]}$
such that
\begin{equation*}\label{TLCP111}
  \mathcal{Z} \geq \mathcal{O},~\mathcal{M}\mathcal{Z}+\mathcal{Q} \geq \mathcal{O},~\text{and}~\langle \mathcal{Z}, \mathcal{M}\mathcal{Z}+\mathcal{Q} \rangle = 0, 
\end{equation*}
where $\mathcal{M}\mathcal{Z}$ is a tensor contraction product defined in \cite{Li2020}, and $\mathcal{O}$ denotes a tensor having all of its entries equal to zero. When $m =1$, the TLCP reduces to the classical linear complementarity problem. The introduction of the TLCP provides a useful framework for extending complementarity theory to tensor spaces and has stimulated further research in this direction. The authors in \cite{SHANG2024115383} introduced the notions of $T$-$R_{0}$ and $T$-$R$ tensors and investigated bounds for the solution set of the TLCP. Later, Shang and Jia \cite{Shang19062025} introduced several structured tensors over tensor spaces, including $T$-$S$ tensors, $T$-semi-positive ($T$-strictly semi-positive) tensors, $T$-copositive ($T$-strictly copositive) tensors and $T$-$P$ tensors and studied the global uniqueness and solvability of the TLCP under these structured tensors. Moreover, it was shown that 
the TLCP can be associated with the following quadratic programming problem:
\begin{equation*}\label{QPP112}
\begin{aligned}
\min \quad & \mathcal{Z}(\mathcal{M}\mathcal{Z}) + \mathcal{Q}\mathcal{Z} \\
\text{s.t.} \quad & \mathcal{M}\mathcal{Z} + \mathcal{Q} \geq \mathcal{O}, 
& \mathcal{Z} \geq \mathcal{O}.
\end{aligned}
\end{equation*} 
It was proved that the feasible set of this quadratic program coincides with that of the TLCP, and any minimizer of the program provides a solution to the TLCP. This equivalence also indicates that the TLCP is closely related to certain quadratic programming problem, thereby highlighting their potential applications in global optimization.

Despite significant advances in the theory of the $\mathrm{TLCP}$, the convexity of the solution set of the $\mathrm{TLCP}$ remains largely unexplored. Motivated by this gap, we study the convexity of the solution set of the $\mathrm{TLCP}$. We introduce the notion of a $T$-column sufficient tensor and study its connection with the structure of the solution set of the  $\mathrm{TLCP}$. Although Li and Huang \cite{Li2020} provided a sufficient condition, it does not fully characterize convexity. In contrast, we establish an equivalent condition. We also investigate the uniqueness of solutions to the $\mathrm{TLCP}$ associated with $T$-column sufficient tensors. The main contributions of this paper are summarized as follows:
\begin{itemize}
\item[\rm(i)]  We introduce the notion of a $T$-column sufficient tensor and study its properties and relationships with $T$-$P$ tensors, $T$-positive semidefinite tensors, and $T$-copositive tensors.
\item[\rm(ii)] We provide an equivalent condition for a $T$-$P$ tensor to be a $T$-column sufficient tensor.
\item[\rm(iii)] We prove that the solution set of the TLCP is convex if and only if the involved tensor is $T$-column sufficient.
\item[\rm(iv)] We provide sufficient conditions for the feasibility implying solvability, and for the uniqueness of the solution to the TLCP.
\end{itemize} 
To provide a clear overview, Figure~\ref{fig} illustrates the relationships among several structured tensors and highlights the contributions of the present work, with the red arrows indicating the results established in this paper.
\begin{figure}[!htbp]
\centering
\begin{tikzpicture}[
    every node/.style={font=\small},
    box/.style={draw, rectangle, rounded corners, minimum width=2.5cm, minimum height=1cm, align=center},
    comb/.style={draw, dashed, rectangle, rounded corners, minimum width=3.2cm, minimum height=1.4cm, align=center},
    arrow/.style={->, thick}
]

\node[box] (PSD)  at (0,0) {$T$-positive\\ semidefinite};
\node[box] (CS) at (4,0) {$T$-column \\sufficient};
\node[box] (ND) at (8,1.5) {$T$-non-\\degenerate};
\node[box] (P)  at (4,2.2) {$T$-$P$};
\node[box] (CCS)  at (4.5,-1.8) {$T$-column sufficient\\ on $\mathbb{R}^{[m,n]}_{+}$};
\node[box] (COP)  at (8,-1.8) {$T$-copositive};
\node[box] (SP) at (4,-3.8) {$T$-semi-positive};

\node[comb] (COMB) at (8,3.1) 
{$T$-column sufficient \\\;+\\\; $T$-non-degenerate};
\node[box] (convexity) at (0,-2.5) {Convexity of \\
${\mathrm{SOL}(\mathcal{M},\mathcal{Q})}$};

\draw[<->,thick, red] (convexity) to (CS);
\draw[arrow, red] (PSD) -- (CS);
\draw[arrow,red] (CS) -- (CCS);
\draw[arrow,red] (COP) -- (CCS);
\draw[arrow,red] (P) to (CS);
\draw[arrow,red] (P) to  (ND);
\draw[arrow,red] (CCS) to (SP);
\draw[<->, thick,red] (P) -- (COMB);
\draw[->, thick] (COP) to [bend left =25]
    node[midway, below=1pt]  {\scriptsize \cite{Shang19062025}} 
(SP);
\draw[->, thick] (PSD) to 
    node[midway, right]{\scriptsize \cite{Li2020}} 
(convexity);
\end{tikzpicture}

\caption{Illustration of the results obtained for a $T$-column sufficient tensor.}\label{fig}
\end{figure}
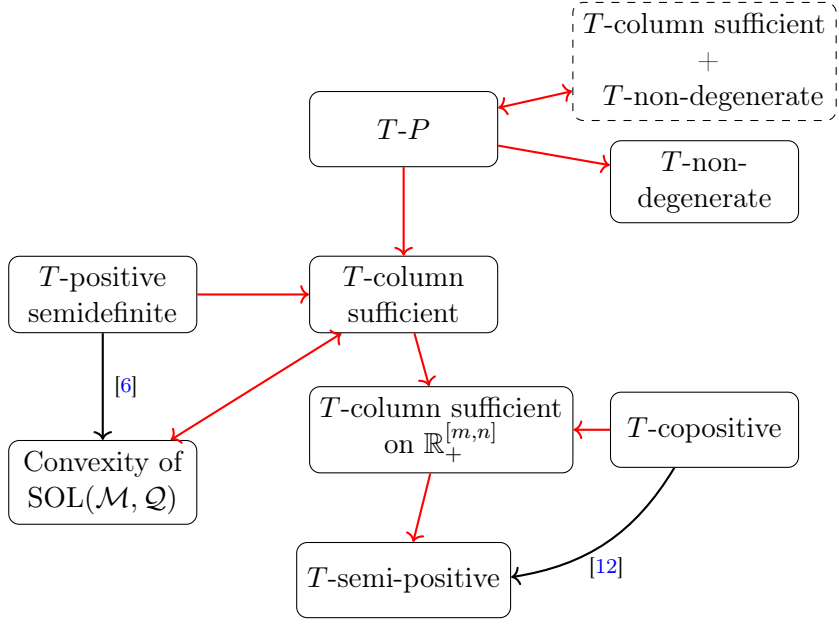

The organization of the paper is as follows. In Section \ref{Pre}, we recall the necessary concepts and results used in the sequel. In Section \ref{SectionT-CS}, we introduce $T$-column sufficient tensors and study their properties and relationships with other structured tensors. In Section \ref{Properties}, we investigate the properties of the solution set of the $\mathrm{TLCP}$, such as convexity in Subsection \ref{Convexity}, feasibility implies solvability in Subsection \ref{FEASOL}, and uniqueness in Subsection \ref{Uniqueness}. Conclusions are given in Section \ref{conclusions}.

\section{Preliminaries}\label{Pre}
In this section, we recall some basic concepts and results that will be used throughout the paper. 
\begin{enumerate}
\item A tensor of appropriate order and dimension with all its entries equal to zero is denoted by $\mathcal{O}.$
\item The set $\{1,2,...,n\}$ is denoted by $[n].$
\item Let $m\geq 2$ and $k_1,k_2,\ldots,k_m$ be positive integers. A real tensor $\mathcal{M}=(m_{i_1\ldots i_m})$ is called an $m$-th order $(k_1\times \cdots \times k_m)$-dimensional tensor, where 
$m_{i_1\ldots i_m}\in\mathbb{R}$ and $i_j\in [k_j]$ 
for $j\in[m]$.
\item The set of all $m$-th order $(k_1\times\cdots\times k_m)$-dimensional real tensors is denoted by $\mathbb{R}^{k_1\times\cdots\times k_m}$.
\item If $k_1=\cdots=k_m=n$, then $\mathcal{M}$ is called an 
$m$-th order $n$-dimensional tensor, and the set of all such tensors is denoted by $\mathbb{R}^{[m,n]}$. 
\item A tensor $\mathcal{M} \in \mathbb{R}^{[m,n]}$ is said to be \emph{nonnegative} (respectively, \emph{positive}) if all of its entries are nonnegative (respectively, positive). In this case, we write $\mathcal{M} \geq \mathcal{O}$ (respectively, $\mathcal{M} > \mathcal{O}$). We denote the sets of all nonnegative and positive tensors in $\mathbb{R}^{[m,n]}$ as $\mathbb{R}^{[m,n]}_+$ and $\mathbb{R}^{[m,n]}_{++}$, respectively.
\item Let $\mathcal{M}, \mathcal{N}$ be in $\mathbb{R}^{k_1\times\cdots\times k_m}$. Then for any $i_{j} \in [k_{j}] (j \in [m])$, the tensors 
$\max\{\mathcal{M},\mathcal{N}\},\mathcal{M}^{+}, \mathcal{M}^{-}$ are in $\mathbb{R}^{k_1\times\cdots\times k_m}$ defined as
\begin{enumerate}
\iffalse
\item $(\min\{\mathcal{M},\mathcal{N}\})_{i_{1}i_{2}...i_{m}} := \min\{m_{i_{1}i_{2}...i_{m}}, n_{i_{1}i_{2}...i_{m}}\},$
\fi
\item $(\max\{\mathcal{M},\mathcal{N}\})_{i_{1}i_{2}...i_{m}} := \max\{m_{i_{1}i_{2}...i_{m}}, n_{i_{1}i_{2}...i_{m}}\},$
    \item $\mathcal{M}^{+} = \max\{\mathcal{M},\mathcal{O}\},$
    \item $\mathcal{M}^{-} = \max\{-\mathcal{M},\mathcal{O}\}.$
\end{enumerate}
In accordance with standard convention, we denote $- \mathcal{M} = (-1)\mathcal{M}.$ It is easy to see that $(-\mathcal{M})^{-} = \mathcal{M}^{+}.$
\item A tensor $\mathcal{M} = (m_{i_{1}i_{2}...i_{2m}}) \in \mathbb{R}^{[2m,n]}$ is said to be a block symmetric tensor (see \cite{Li2020}) if $m_{i_{1}...i_{m}i_{m+1}...i_{2m}} = m_{i_{m+1}...i_{2m}i_{1}...i_{m}}$ for any $i_{j} \in [n]~(j \in [2m]).$
\iffalse
{\item A tensor $\mathcal{M} = (m_{i_{1}i_{2}...i_{2m}}) \in \mathbb{R}^{[2m,n]}$ is said to be a block diagonal tensor (see \cite{Li2020}) if $m_{i_{1}...i_{m}i_{m+1}...i_{2m}}=0$ for all $(i_{1},i_{2},...,i_{m}) \neq (i_{m+1},i_{m+2},...,i_{2m}).$
$\mathcal{M}$ can have nonzero entries only when $(i_{1},i_{2},...,i_{m}) = (i_{m+1},i_{m+2},...,i_{2m}).$}
\fi
\item Let $\mathcal{M}, \mathcal{N}$ be  in $\mathbb{R}^{k_1\times\cdots\times k_m}$ and $\beta \in \mathbb{R}$. Then for any $i_{j} \in [k_{j}] (j \in [m])$, we define $\beta \mathcal{M}$, $\mathcal{M}+\mathcal{N}$ and the inner product $\langle \mathcal{M}, \mathcal{N} \rangle$ as
$\beta \mathcal{M}:= (\beta m_{i_{1}i_{2}...i_{m}}),$
$\mathcal{M} + \mathcal{N} := (m_{i_{1}i_{2}...i_{m}}+n_{i_{1}i_{2}...i_{m}}),$ and
\begin{equation}\label{innerproduct}
   \langle \mathcal{M}, \mathcal{N} \rangle := 
\displaystyle{}{}\sum_{i_1=1}^{k_1}\cdots\sum_{i_m=1}^{k_m}
m_{i_1\ldots i_m}
n_{i_{1}\ldots i_m}.
\end{equation}
Therefore, $\mathbb{R}^{k_1\times\cdots\times k_m}$ becomes a real Hilbert space, and we call it a tensor space due to its elements being tensors. The inner product induces a norm of a tensor $\mathcal{M},$ which is
\begin{equation*}
    \|\mathcal{M}\| = \sqrt{\langle \mathcal{M},\mathcal{M}\rangle} = \sqrt{\displaystyle{}{}\sum_{i_1=1}^{k_1}\cdots\sum_{i_m=1}^{k_m}
m^{2}_{i_1\ldots i_m}}.
\end{equation*}
\end{enumerate}
In the following, we recall the concept of $J(p)$-mode product of tensors from \cite{Shang19062025}, which was originally introduced in \cite{Li2020}.

\begin{definition}\label{contraction}\rm \citep[Definition 2.1]
{Shang19062025}
Let $\mathcal{M}$ be a tensor in $\mathbb{R}^{k_{1}\times\cdots \times k_{m}}$, and $\mathcal{N}$ be a tensor in $\mathbb{R}^{{\bar{k}_{1}}\times \cdots \times {\bar{k}_{p}}},$ where $p \leq m$. Suppose that there exists $p$ monotonically increasing different integers $\hat{j}_{1},\ldots,\hat{j}_{p} \in [m]$ such that $k_{\hat{j}_{l}}=\bar{k}_{l}$ for any $l \in [p]$. Let $J(p):=\{\hat{j}_{1},\ldots,\hat{j}_{p}\}$ and $I(p):= \{\hat{i}_{1},\ldots,\hat{i}_{m-p}\}$ be a partition of the set $[m]$ such that $J(p) \cap I(p) = \emptyset$ and $J(p) \cup I(p) = [m],$ where the elements in $I(p)$ are monotonically increasing. For any $i_{\hat{i}_{1}} \in [k_{\hat{i}_{1}}],\ldots, i_{\hat{i}_{m-p}}\in [k_{\hat{i}_{m-p}}]$, the $J(p)$-mode product of tensors $\mathcal{M}$  and $\mathcal{N}$  is denoted by $\mathcal{M}_{\times J(p)} \mathcal{N},$ which is a tensor in $\mathbb{R}^{k_{\hat{i}_{1}} \times \ldots \times k_{\hat{i}_{m-p}}}$ with 
\begin{equation*}
(\mathcal{M}\times_{J(p)}\mathcal{N})_{i_{\hat{i}_{1}}\ldots i_{\hat{i}_{m-p}}} =\sum_{i_{\hat{j}_{1}}=1}^{\bar{k}_{1}}
\cdots\sum_{i_{\hat{j}_{p}}=1}^{\bar{k}_{p}}
m_{sort}\{{i_{\hat{i}_1},\ldots,i_{\hat{i}_{m-p}},i_{\hat{j}_{1}},\ldots,i_{\hat{j}_{p}}}\}
n_{i_{\hat{j}_{1}}\ldots i_{\hat{j}_{p}}},
\end{equation*}
where ${sort}\{{i_{\hat{i}_1},\ldots,i_{\hat{i}_{m-p}},i_{\hat{j}_{1}},\ldots,i_{\hat{j}_{p}}}\}$ indicates that all elements are sorted in natural number order of their subscripts.
\end{definition}
\noindent The following observations follow directly from the definition of the $J(p)$-mode product of tensors.
\begin{enumerate}
    \item If $m = p$, then the $J(p)$-mode product becomes the inner product  defined in Equation (\ref{innerproduct}).
    \item If $J(p) = [m] \setminus [m-p],$ then $I(p) = [m-p].$ For any $r \in I(p)$ and $s \in J(p)$, we get $\hat{i}_{r} = r$ and $\hat{j}_{s} = s$. So, for any $i_{r} \in [k_{r}](r \in I(p))$,
    \begin{equation*}
(\mathcal{M}\times_{J(p)}\mathcal{N})_{{{i}_{1}}\ldots {{i}_{m-p}}} =\sum_{i_{m-p+1=1}}^{\bar{k}_{1}}
\cdots\sum_{i_{m}=1}^{\bar{k}_{p}}
m_{i_{1} \ldots i_{{m-p}}i_{m-p+1} \ldots i_{m}}
n_{i_{m-p+1} \ldots i_{m}}.
    \end{equation*}
It is easy to see that $\mathcal{M}_{\times J(p)} \mathcal{N}$ is a tensor in $\mathbb{R}^{r_{1} \times \ldots r_{m-p}}.$ In this case, we use $\mathcal{M}\mathcal{N}$ to denote $\mathcal{M}_{\times J(p)} \mathcal{N}$.
\iffalse
\item Let $\mathcal{M} \in \mathbb{R}^{[2p,n]}$ and $\mathcal{N} \in \mathbb{R}^{[p,n]}$. Consider $J(p) = [2p] \setminus [2p-p] = \{p+1,p+2,...,2p\}$ and $I(p) = [2p-p] = \{1,2,...p\}$ then for any $i_{j} \in [n] (j \in [p])$ the $J(p)$-mode product of $\mathcal{M}$ and $\mathcal{N}$ is given as
\begin{equation*}
(\mathcal{M}\times_{J(p)}\mathcal{N})_{{{i}_{1}}\ldots {{i}_{p}}} =\sum_{i_{p+1=1}}^{n}
\cdots\sum_{i_{2p}=1}^{{n}}
m_{i_{1} \ldots i_{{p}}i_{p+1} \ldots i_{2p}}
n_{i_{p+1} \ldots i_{2p}}.  
\end{equation*}
\fi
   \item When $m=2$ and $p=1$, the $J(p)$-mode product of $\mathcal{M}$ and $\mathcal{N}$ reduces to the product between a matrix and a vector.
\end{enumerate}
Throughout the paper, unless otherwise specified, we consider only the $J(p)(=[m] \setminus [m-p])$-mode product of two tensors for a given $p \in [m]$.

We recall the following lemmas and definitions from \cite{Shang19062025}, which will be useful in the sequel.

\begin{lemma} {\rm \citep[Lemma 2.1]{Shang19062025}}
Let $\mathcal{M},\mathcal{N}$ be in $\mathbb{R}^{k_{1}\times\cdots \times k_{m}}$, and 
$\mathcal{X},\mathcal{Y}$ be $p$-th order and  $({k_{m-p+1}\times\cdots \times k_{m})}$-dimensional real tensors ($p \in [m]$), and $\alpha\in\mathbb{R}$.  Then the following relations hold:
$(\mathcal{M}+\mathcal{B})\mathcal{X}=\mathcal{M}\mathcal{X}+\mathcal{B}\mathcal{X},$ $\mathcal{M}(\mathcal{X}+\mathcal{Y})=\mathcal{M}\mathcal{X}+\mathcal{M}\mathcal{Y},$
$(\alpha\mathcal{M})\mathcal{X}=\alpha(\mathcal{M}\mathcal{X}),$ $ \mathcal{M}(\alpha\mathcal{X})=\alpha(\mathcal{M}\mathcal{X}).$
\end{lemma}

\begin{definition}\rm\label{TLCPDefinition}\cite{Shang19062025} For any $\mathcal{M} \in \mathbb{R}^{[2m,n]}$ and $\mathcal{Q} \in \mathbb{R}^{[m,n]}$, the linear complementarity problem over tensor space (for short, TLCP) is to seek a tensor $\mathcal{Z} \in \mathbb{R}^{[m,n]}$ such that
\begin{equation}\label{TLCPequation}
\mathcal{Z} \in \mathbb{R}^{[m,n]}_{+},~\mathcal{M}\mathcal{Q}+\mathcal{Z} \in \mathbb{R}^{[m,n]}_{+}~\text{and}~\langle \mathcal{Z},\mathcal{M}\mathcal{Q}+\mathcal{Z}\rangle =0.
\end{equation}
This problem is denoted by $\mathrm{TLCP}(\mathcal{M},\mathcal{Q})$, and the set of all  tensors $\mathcal{Z} \in \mathbb{R}^{[m,n]}$ satisfying Equation (\ref{TLCPequation}) is said to be the solution set of the $\mathrm{TLCP}(\mathcal{M},\mathcal{Q})$, denoted by $\mathrm{SOL}(\mathcal{M},\mathcal{Q})$.
\end{definition}

\noindent Note that when $m=1$, then the TLCP reduces to the classical linear complementarity problem \cite{MR3396730}.

\begin{definition}\rm \citep[Definition 2.2]{Shang19062025}
Let $\mathcal{M}$ be in $\mathbb{R}^{[2m,n]}$ and 
$\mathcal{Q}\in\mathbb{R}^{[m,n]}$.
The $\mathrm{TLCP}(\mathcal{M},\mathcal{Q})$
is said to be feasible if and only if there exists 
$\mathcal{Z}\in\mathbb{R}^{[m,n]}_{+}$ such that $\mathcal{M}\mathcal{Z}+\mathcal{Q}\in\mathbb{R}^{[m,n]}_+.$
\end{definition}

\begin{definition} \rm \citep[Definition 2.3]{Shang19062025}
Let $\mathcal{M}=(m_{i_1\ldots i_m})\in\mathbb{R}^{[m,n]}$ 
and $k\in[n]$. If $m^{(k)}_{i_1\ldots i_m}=m_{i_1\ldots i_m},
~\text{for any}~\, i_1,\ldots,i_m\in[k],
$ then the tensor $\mathcal{M}^{(k)}=(m^{(k)}_{i_1\ldots i_m}) \in \mathbb{R}^{[m,k]}$ is called a $k$-th sequential principal subtensor of $\mathcal{M}$.
\end{definition}

\begin{lemma} \rm{\citep[Lemma 2.2]{Shang19062025}}
Let $\mathcal{M}=(m_{i_1\ldots i_{2m}})\in\mathbb{R}^{[2m,n]}$ and 
$\mathcal{Z}\in\mathbb{R}^{[m,n]}$. 
If $\mathcal{M}$ is block symmetric, then $
\nabla \big[\mathcal{Z}(\mathcal{M}\mathcal{Z}) \big]
=2\mathcal{M}\mathcal{Z}.$
\end{lemma}

\begin{lemma} \rm{\citep[Lemma 2.3]{Shang19062025}}
Let $\mathcal{M} \in \mathbb{R}^{[2m,n]}$ and $\mathcal{Z} = (z_{i_{1}i_{2}...i_{m}}) \in \mathbb{R}^{[m,n]}$. Then $\nabla (\mathcal{M} \mathcal{Z}) = \mathcal{M}$ and $\nabla \bigg(\displaystyle{}{} \sum_{i_{1},...,i_{m}=1}^{n} z^{2}_{i_{1}...i_{m}}\bigg) = 2 \mathcal{Z
}.$
\end{lemma}

\begin{definition}\rm\cite{Li2020,Shang19062025,SHANG2024115383} Let $\mathcal{M} \in \mathbb{R}^{[2m,n]}$. Then $\mathcal{M}$ is said to be a
\begin{enumerate}
\iffalse
\item[\rm(i)] $T$-$S$ tensor if and only if $\mathrm{TLCP}(\mathcal{M},\mathcal{Q})$ is feasible for each $\mathcal{Q} \in \mathbb{R}^{[m,n]}.$
\fi
\item[\rm(i)] $T$-copositive if $\mathcal{Z}(\mathcal{M}\mathcal{Z}) \geq {0}$  for each $\mathcal{Z} \in \mathbb{R}^{[m,n]}_{+}.$
\item[\rm(ii)] $T$-strictly copositive if $\mathcal{Z}(\mathcal{M}\mathcal{Z}) > {0}$  for each $\mathcal{Z} \in \mathbb{R}^{[m,n]}_{+}\setminus\{\mathcal{O}\}.$
\item[\rm(iii)] $T$-semi-positive if for any $\mathcal{Z} \in \mathbb{R}^{[m,n]}_{+}\setminus \{\mathcal{O}\}$, there exist $i_{i},i_{2},...,i_{m} \in [n]$ such that
$z_{i_{1}i_{2}...i_{m}} > 0,~\text{and}~(\mathcal{M}\mathcal{Z})_{i_{1}i_{2}...i_{m}} \geq 0.$
\item[\rm(iv)] $T$-strictly semi-positive if for any $\mathcal{Z} \in \mathbb{R}^{[m,n]}_{+}\setminus \{\mathcal{O}\}$, there exist $i_{i},i_{2},...,i_{m}$ in $[n]$ such that
$z_{i_{1}i_{2}...i_{m}} > 0,~\text{and}~(\mathcal{M}\mathcal{Z})_{i_{1}i_{2}...i_{m}} > 0.$
\item[\rm(v)] $T$-$P$ tensor if and only if for any $\mathcal{Z} \in \mathbb{R}^{[m,n]}\setminus\{\mathcal{O}\}$, there exist $i_{1},i_{2},...,i_{m} \in [n]$  such that $z_{i_{1}i_{2}...i_{m}}(\mathcal{M}\mathcal{Z})_{i_{1}i_{2}...i_{m}} > 0.$
\iffalse
\item[\rm(vi)] $T$-$R_{0}$ tensor if and only if the $\mathrm{TLCP}(\mathcal{M},\mathcal{O})$ has a unique solution (zero tensor).
\fi
\item[\rm(vi)] $T$-positive semidefinite if $\mathcal{Z}(\mathcal{M}\mathcal{Z}) \geq 0$ for any $\mathcal{Z} \in \mathbb{R}^{[m,n]}.$ 
\end{enumerate}
\end{definition}

\begin{proposition}\label{convexity}\rm{\citep[Proposition 4.9]{Li2020}} 
For any $\mathcal{M} \in \mathbb{R}^{[2m,n]}$ and $\mathcal{Q} \in \mathbb{R}^{[m,n]}$, the following statements are equivalent.
\begin{enumerate}
    \item[\rm(i)] The set $\mathrm{SOL}(\mathcal{M},\mathcal{Q})$ is a convex set.
    \item[\rm(ii)] For any $\mathcal{Z}^{1}, \mathcal{Z}^{2}$ in $\mathrm{SOL}(\mathcal{M},\mathcal{Q})$, the following equation holds.
    \begin{equation}\label{cross commutative}
      \langle \mathcal{Z}^{1}, \mathcal{M}\mathcal{Z}^{2}+\mathcal{Q} \rangle =  \langle \mathcal{Z}^{2}, \mathcal{M}\mathcal{Z}^{1}+\mathcal{Q} \rangle = 0.
    \end{equation}
\end{enumerate}
\end{proposition}

\begin{theorem}\label{TSP}\rm{\citep[Theorem 4.1]{Shang19062025}} A tensor $\mathcal{M} \in \mathbb{R}^{[2m,n]}$ is $T$-semi-positive if and only if the $\mathrm{TLCP}(\mathcal{M},\mathcal{Q})$ has a unique solution for every $\mathcal{Q} \in \mathbb{R}^{[m,n]}_{++}.$
\end{theorem}

\begin{lemma}\label{QPP1}{\rm \citep[Lemma 4.1]{Shang19062025}} 
Consider the following quadratic programming problem:
\begin{equation}\label{QPP}
\begin{aligned}
\min \quad & \mathcal{Z}(\mathcal{M}\mathcal{Z}) + \mathcal{Q}\mathcal{Z} \\
\text{\rm {s.t.}} \quad & \mathcal{M}\mathcal{Z} + \mathcal{Q} \in \mathbb{R}^{[m,n]}_{+}, 
& \mathcal{Z} \in \mathbb{R}^{[m,n]}_{+}.
\end{aligned}
\end{equation}
Then the feasible set of the $\mathrm{TLCP}(\mathcal{M},\mathcal{Q})$ is also the feasible set of (\ref{QPP}). 
If $\mathcal{M} =(m_{i_{1}i_{2}...i_{2m}}) \in \mathbb{R}^{[2m,n]}$ is a block symmetric tensor and $\mathrm{TLCP}(\mathcal{M},\mathcal{Q})$ is feasible, then the quadratic program (\ref{QPP}) has an optimal solution, $\mathcal{Z}^{*}$. Moreover,  there exists a multiplier $\mathcal{U}^{*}$ satisfying the conditions:
\begin{eqnarray}
& 2\mathcal{M}\mathcal{Z}^{*} + \mathcal{Q} - \mathcal{M}\mathcal{U}^{*} \in \mathbb{R}^{[m,n]}_{+}, \label{KKT1}\\
& \langle \mathcal{Z}^{*}, 2\mathcal{M}\mathcal{Z}^{*} + \mathcal{Q} - \mathcal{M}\mathcal{U}^{*} \rangle = 0, \label{KKT2}\\
& \mathcal{U}^{*} \in \mathbb{R}^{[m,n]}_{+}, \langle \mathcal{U}^{*}, \mathcal{M}{\mathcal{Z}^{*}} + \mathcal{Q} \rangle = 0, \label{KKT3}\\
& \mathcal{Z}^{*} \in \mathbb{R}^{[m,n]}_{+},~ \mathcal{M}\mathcal{Z}^{*} + \mathcal{Q} \in \mathbb{R}^{[m,n]}_{+}. \label{KKT4}
\end{eqnarray}
Finally, the tensors $\mathcal{Z}^{*}$ and $\mathcal{U}^{*}$ satisfy
\begin{equation}\label{KKT5}
   (\mathcal{Z}^{*}-\mathcal{U}^{*})_{i_{1}i_{2}...i_{m}}\big[\mathcal{M}(\mathcal{Z}^{*}-\mathcal{U}^{*})\big]_{i_{1}i_{2}...i_{m}} \leq {0},~\text{for all}~ i_{1},i_{2},...,i_{m} \in [n].
\end{equation}
\end{lemma}

\section{\texorpdfstring{$T$}{T}-column sufficient tensor}\label{SectionT-CS}
In this section, we first define a $T$-column sufficient tensor, which is a natural extension of the column sufficient matrix from $\mathbb{R}^{n}$ to $\mathbb{R}^{[2m,n]}.$
\begin{definition}\rm\label{T-CSdefinition} A tensor $\mathcal{M} \in \mathbb{R}^{[2m,n]}$  is said to be a 
\begin{enumerate}
\item[\rm(i)] $T$-column sufficient tensor on $\mathbb{K} \subseteq \mathbb{R}^{[m,n]}$ if and only if for any $\mathcal{Z} \in \mathbb{K},$ 
\begin{equation*}\label{T-CCSequation}
{z}_{i_{1}i_{2}...i_{m}}(\mathcal{M}\mathcal{Z})_{i_{1}i_{2}...i_{m}}
\leq {0},~{\forall}~ i_{j} \in [n] \implies  
{z}_{i_{1}i_{2}...i_{m}}(\mathcal{M}\mathcal{Z})_{i_{1}i_{2}...i_{m}} = {0},~{\forall}~ i_{j}\in [n]. 
\end{equation*}
\item[\rm(ii)]  $T$-column sufficient tensor if and only if for any $\mathcal{Z} \in \mathbb{R}^{[m,n]},$
\begin{equation*}\label{T-CSequation}
{z}_{i_{1}i_{2}...i_{m}}(\mathcal{M}\mathcal{Z})_{i_{1}i_{2}...i_{m}}
\leq {0},~{\forall}~ i_{j}\in [n] \implies  
{z}_{i_{1}i_{2}...i_{m}}(\mathcal{M}\mathcal{Z})_{i_{1}i_{2}...i_{m}} = {0},~{\forall}~ i_{j}\in [n]. 
\end{equation*}
\end{enumerate}
\end{definition}

\begin{remark}\rm
It is evident from the above definition that when $m=1$, then the definition of a $T$-column sufficient tensor coincides with a column sufficient matrix \citep[Definition 3.5.1]{MR3396730}. Moreover, any $T$-column sufficient tensor is a $T$-column sufficient tensor on $\mathbb{K} \subseteq \mathbb{R}^{[m,n]}$, and the converse is not true in general (see Example \ref{T-CS on K but not T-CS}). 
\end{remark}

\begin{remark}\rm
Chen et al. \citep[Definition 2]{MR3778365} introduced the concept of a column sufficient tensor as a generalization of the column sufficient matrix to the tensor setting and established several spectral and structural properties of column sufficient tensors. Although both the column sufficient tensor and the structured tensor proposed in our work extend the notion of a column sufficient matrix to tensors, the two concepts are completely different. In our definition (see Definition \ref{T-CSdefinition}), we consider the tensor $\mathcal{Z}$ to be in the tensor space $\mathbb{R}^{[m,n]}$. In contrast, in the definition of the column sufficient tensor given by Chen et al. \cite{MR3778365}, the vector  $\mathbf{z}$ is an element of the vector space $\mathbb{R}^{n}$.  To clearly reflect this distinction, we refer to our newly defined structured tensor as the $T$-column sufficient tensor, where $T$ indicates the tensor space.
\end{remark}   

In the following, we provide some examples to illustrate the Definition \ref{T-CSdefinition}.
\begin{example}\rm
Let $\mathcal{M}=(m_{i_{1}i_{2}...i_{2m}})$  be a tensor in $\mathbb{R}^{[2m,n]}$ with entries  
\begin{equation*}
    m_{i_{1}...i_{m}i_{m+1}...i_{2m}} = \begin{cases}
       \geq {0},~\text{if}~(i_{1},i_{2},...,i_{m}) = (i_{m+1},i_{m+2},...,i_{2m}),\\
       0,~\text{else}.
    \end{cases}
\end{equation*}
Then $\mathcal{M}$ is $T$-column sufficient. Indeed, if any $\mathcal{Z} =(z_{i_{1}i_{2}...i_{m}}) \in \mathbb{R}^{[m,n]}$ satisfies
\begin{equation*}
    z_{i_{1}i_{2}...i_{m}}(\mathcal{M}\mathcal{Z})_{i_{1}i_{2}...i_{m}} \leq {0},~\text{for all}~i_{1},i_{2},...i_{m} \in [n],
\end{equation*}
then $m_{i_{1}i_{2}...i_{m}i_{1}i_{2}...i_{m}}z^{2}_{i_{1}i_{2}...i_{m}} \leq {0}~\text{for all}~i_{1},i_{2},...i_{m} \in [n]$. Since $m_{i_{1}i_{2}...i_{m}i_{1}i_{2}...i_{m}} \geq {0},$ for all $i_{j} \in [n]~(j \in [m])$, therefore $m_{i_{1}i_{2}...i_{m}i_{1}i_{2}...i_{m}}z^{2}_{i_{1}i_{2}...i_{m}}= {0}~\text{for all}~i_{1},i_{2},...i_{m} \in [n].$ This implies that $z_{i_{1}i_{2}...i_{m}}(\mathcal{M}\mathcal{Z})_{i_{1}i_{2}...i_{m}} = {0},~\text{for all}~i_{1},i_{2},...i_{m} \in [n].$ Thus $\mathcal{M}$ is $T$-column sufficient.
\end{example}

\begin{example}\rm\label{T-CS on K but not T-CS}
Let $\mathcal{M} = (m_{i_{1}i_{2}i_{3}i_{4}}) \in \mathbb{R}^{[4,2]},$ where $m_{1111} = m_{1212} =m_{2122} = 2, m_{2211}=m_{2112}=m_{2222}=1$ and other entries be zero. We claim that \rm(i) $\mathcal{M}$ is $T$-column sufficient on $\mathbb{R}^{[2,2]}_{+}$,  \rm(ii) but not $T$-column sufficient. For any $\mathcal{Z} \in \mathbb{R}^{[2,2]}$, we have 
\begin{equation*}
\mathcal{M}\mathcal{Z} =   \begin{pmatrix}
2z_{11} & 2z_{21} \\
2z_{22}+z_{12} & z_{11}+z_{22}
\end{pmatrix}.
\end{equation*}
\begin{itemize}
\item[\rm(i)] Let $z_{ij}(\mathcal{M}\mathcal{Z})_{ij} \leq {0},$ for all $i,j \in [2],$ where $\mathcal{Z} \in \mathbb{R}^{[2,2]}_{+}.$ This implies that 
\begin{equation*}
  2z_{11}^{2} \leq {0},2z_{12}z_{21} \leq {0}, z_{21}(2z_{22}+z_{12})\leq {0},~\text{and}~z_{22}(z_{11}+z_{22}) \leq {0}.      
\end{equation*}
From this, we obtain $z_{11}=0, z_{22}=0$, and $z_{12}z_{21} = 0.$ Thus $z_{ij}(\mathcal{M}\mathcal{Z})_{ij} =0,$ for all $i,j \in [2],$ and for each $\mathcal{Z} \in \mathbb{R}^{[2,2]}_{+}.$ Hence $\mathcal{M}$ is $T$-column sufficient on $\mathbb{R}^{[2,2]}_{+}.$ 
\item[\rm(ii)] Let $\mathcal{Z}= \begin{pmatrix}
    0 & 1 \\
    -1 & 0
\end{pmatrix} \in \mathbb{R}^{[2,2]}.$ Then, we have $z_{11}(\mathcal{M}\mathcal{Z})_{11} = 0, z_{12}(\mathcal{M}\mathcal{Z})_{12} = -2, z_{21}(\mathcal{M}\mathcal{Z})_{21} = -1, z_{22}(\mathcal{M}\mathcal{Z})_{22} = 0.$ It follows from the definition of a $T$-column sufficient tensor that $\mathcal{M}$ is not $T$-column sufficient. 
\end{itemize}
\end{example}

The following proposition characterizes the relationship between $T$-column sufficient tensors and other structured tensors.

\begin{proposition}\label{T-P is T-CS} 
Let $\mathcal{M} \in \mathbb{R}^{[2m,n]}$.
\begin{enumerate}
\item[\rm(i)] If $\mathcal{M}$ is a $T$-positive semidefinite tensor, then it  is $T$-column sufficient.
\item[\rm(ii)] If $\mathcal{M}$ is a $T$-$P$ tensor, then $\mathcal{M}$ is $T$-column sufficient.
\item[\rm(iii)]  If $\mathcal{M}$ is  $T$-copositive, then it is $T$-column sufficient on $\mathbb{R}^{[m,n]}_{+}.$
\end{enumerate}
\end{proposition}
\begin{proof}
\rm(i):  Let $\mathcal{M} \in \mathbb{R}^{[2m,n]}$ be $T$-positive semidefinite. This implies that $\mathcal{Z}(\mathcal{M}\mathcal{Z}) \geq {0}$ for any $\mathcal{Z} \in \mathbb{R}^{[m,n]}.$ Suppose that there exists  $\mathcal{Z}$\!\! in $\mathbb{R}^{[m,n]}$ satisfying ${z}_{i_{1}i_{2}...i_{m}}(\mathcal{M}\mathcal{Z})_{i_{1}i_{2}...i_{m}}\\
\leq {0},~{\forall}~ i_{j}\in [n]$. This gives $\mathcal{Z}(\mathcal{M}\mathcal{Z}) \leq {0}$, and therefore $\mathcal{Z}(\mathcal{M}\mathcal{Z}) =0$. Due to each component of $\mathcal{Z}(\mathcal{M}\mathcal{Z})$ being nonpositive, we get ${z}_{i_{1}i_{2}...i_{m}}(\mathcal{M}\mathcal{Z})_{i_{1}i_{2}...i_{m}}
= {0},~{\forall}~ i_{j}\in [n]$. Thus $\mathcal{M}$ is $T$-column sufficient.\\
\rm(ii): Let $\mathcal{M} \in \mathbb{R}^{[2m,n]}$ be a $T$-$P$ tensor. Suppose that $\mathcal{Z} \in \mathbb{R}^{[m,n]}$ satisfies
 \begin{equation}\label{T-P imply T-CS}
  z_{i_{1}i_{2}...i_{m}} (\mathcal{M}\mathcal{Z})_{i_{1}i_{2}...i_{m}} \leq {0},~\text{for all}~i_{1},i_{2},...,i_{m} \in [n].   
 \end{equation}
If $\mathcal{Z} \neq \mathcal{O},$ then we have a nonzero tensor $\mathcal{Z} \in \mathbb{R}^{[m,n]}$  satisfying Equation (\ref{T-P imply T-CS}), which is a contradiction to $\mathcal{M}$ being a $T$-$P$ tensor. Hence, we must have $\mathcal{Z} = \mathcal{O}$, and therefore $z_{i_{1}i_{2}...i_{m}} (\mathcal{M}\mathcal{Z})_{i_{1}i_{2}...i_{m}} ={0},~\text{for all}~i_{1},i_{2},...,i_{m} \in [n].$ Thus $\mathcal{M}$ is a $T$-column sufficient tensor.\\
\rm(iii): The proof follows similarly as $\rm(i)$.
\end{proof}

The converse of the above statements in Proposition \ref{T-P is T-CS} is not true in general, as demonstrated by the following examples.

\iffalse
\begin{proposition}
Any $T$-copositive tensor is $T$-column sufficient on $\mathbb{R}^{[m,n]}_{+},$ but the converse may not be true (refer to Example \ref{T-COP but not T-CCS}). 
\end{proposition}

\begin{proposition}
 However, a $T$-column sufficient tensor may not be a $T$-$P$ tensor (refer to Example \ref{T-CS but not T-P}). 
\end{proposition}
\begin{proof}
 
\end{proof}
\fi

\begin{example}\rm\label{T-CS but not T-PSD}
Let $\mathcal{M} = (m_{i_{1}i_{2}i_{3}i_{4}}) \in \mathbb{R}^{[4,2]},$ where $m_{2111} = m_{2222} =1, m_{1121}=-2$ and other entries be zero. We claim that \rm(i) $\mathcal{M}$ is $T$-column sufficient,  \rm(ii) but not $T$-positive semidefinite. For any $\mathcal{Z} \in \mathbb{R}^{[2,2]}$, we have 
\begin{equation*}
  \mathcal{M}\mathcal{Z} =   \begin{pmatrix}
-2z_{21} & 0 \\
z_{11} & z_{22}
\end{pmatrix}, ~\text{and}~\mathcal{Z}(\mathcal{M}\mathcal{Z}) = -z_{11}z_{21}+z_{22}^{2}.
\end{equation*}
\begin{itemize}
\item[\rm(i)] Let ${z}_{ij}(\mathcal{M}\mathcal{Z})_{ij} \leq {0},$ for all $i,j \in [2].$ This implies  that $-2z_{11}z_{21} \leq 0, 
 z_{11}z_{21} \leq 0,~\text{and}~
 z_{22}^{2} \leq 0$. From this, we obtain $z_{11}z_{21}=0,$ and $z_{22}=0,$ and hence ${z}_{ij}(\mathcal{M}\mathcal{Z})_{ij} = {0},$ for all $i,j \in [2]$. Thus $\mathcal{M}$ is $T$-column sufficient.
 \item[\rm(ii)] Let $\mathcal{Z} = \begin{pmatrix}
     1 & 0 \\
     1 & 0
 \end{pmatrix} \in \mathbb{R}^{[2,2]}.$ Then $\mathcal{Z}(\mathcal{M}\mathcal{Z}) = -1 < 0.$ Thus $\mathcal{M}$ is not $T$-positive semidefinite.
\end{itemize}
\end{example}

\begin{example}\rm\label{T-COP but not T-CCS}
Let $\mathcal{M} = (m_{i_{1}i_{2}i_{3}i_{4}}) \in \mathbb{R}^{[4,2]},$ where $m_{1112} = -10,  m_{1212} = m_{2122}= m_{2211} =1$ and other entries be zero. We claim that \rm(i) $\mathcal{M}$ is $T$-column sufficient on $\mathbb{R}^{[2,2]}_{+}$,  \rm(ii) but not $T$-copositive. For any $\mathcal{Z} \in \mathbb{R}^{[2,2]}$, we have 
\begin{equation*}
\mathcal{M}\mathcal{Z} =   \begin{pmatrix}
-10z_{12} & z_{12}\\
z_{22} & z_{11}
\end{pmatrix}.
\end{equation*} 
\begin{itemize}
\item[\rm(i)] Let $\mathcal{Z} = \begin{pmatrix}
z_{11} & z_{12} \\
z_{21} & z_{22}
\end{pmatrix} \in \mathbb{R}^{[2,2]}_{+}$ satisfies
$z_{ij}(\mathcal{M}\mathcal{Z})_{ij} \leq {0},$ for all $i,j \in [2].$ This implies that
\begin{equation*}
   -10z_{11}z_{12} \leq {0},z_{12}^{2} \leq {0},z_{21}z_{22} \leq {0},~\text{and}~z_{11}z_{22} \leq {0}.     
\end{equation*}
From the above equation, it is easy to verify that $z_{ij}(\mathcal{M}\mathcal{Z})_{ij} = 0,$ for all $i,j \in [2]$. Thus $\mathcal{M}$ is $T$-column sufficient on $\mathbb{R}^{[2,2]}_{+}.$
\item[\rm(ii)] Let $\mathcal{Z} = \begin{pmatrix}
 1 & 1 \\
 1 & 1
\end{pmatrix} \in \mathbb{R}^{[2,2]}_{+}.$ Then $\mathcal{M}\mathcal{Z} = \begin{pmatrix}
  -10 & 1 \\
  1 & 1
\end{pmatrix}.$ Thus, we obtain $\mathcal{Z}(\mathcal{M}\mathcal{Z}) = -7 < 0.$ Therefore $\mathcal{M}$ is not $T$-copositive. 
\end{itemize}
\end{example}

\begin{example}\rm\label{T-CS but not T-P}
Let $\mathcal{M} = (m_{i_{1}i_{2}i_{3}i_{4}}) \in \mathbb{R}^{[4,2]},$ where $m_{1112} = -10, m_{1211} = m_{1212} = m_{2122} =1, m_{2221}=-1$ and other entries be zero. We claim that \rm(i) $\mathcal{M}$ is $T$-column sufficient,  \rm(ii) but not $T$-$P$. For any $\mathcal{Z} \in \mathbb{R}^{[2,2]}$, we have 
\begin{equation*}
\mathcal{M}\mathcal{Z} =   \begin{pmatrix}
-10z_{12} & z_{11}+z_{12}\\
z_{22} & -z_{21}
\end{pmatrix}.
\end{equation*} 
\begin{itemize}
\item[\rm(i)] Let ${z}_{ij}(\mathcal{M}\mathcal{Z})_{ij} \leq {0},$ for all $i,j \in [2].$ This implies  that $$-10z_{11}z_{12} \leq 0, z_{12}(z_{11}+z_{12}) \leq 0,z_{21}z_{22} \leq 0,~\text{and}~ -z_{22}z_{21} \leq 0.$$
From this, we obtain $z_{12}=0,$ and $z_{21}z_{22}=0,$ and hence ${z}_{ij}(\mathcal{M}\mathcal{Z})_{ij} = {0},$ for all $i,j \in [2]$. Thus $\mathcal{M}$ is $T$-column sufficient.
 \item[\rm(ii)] Let $\mathcal{Z} = \begin{pmatrix}
     2 & 0 \\
     0 & 0
 \end{pmatrix} \in \mathbb{R}^{[2,2]}.$ Then $\mathcal{M}\mathcal{Z} = \begin{pmatrix}
     0 & 2 \\
     0 & 0
 \end{pmatrix}.$ Thus, we obtain $z_{ij}(\mathcal{M}\mathcal{Z})_{ij} =0$, for all $i,j \in [2].$ Thus $\mathcal{M}$ is not $T$-$P$.
\end{itemize}
\end{example}

\begin{proposition}
 Any $k$-th sequential principal subtensor of a $T$-column sufficient tensor is $T$-column sufficient.  
\end{proposition}
\begin{proof}
Let $\mathcal{M}^{(k)} = (m_{i_{1}i_{2}...i_{m}}^{(k)}) \in \mathbb{R}^{[m,k]}$ be a $k$-th sequential principal subtensor of a $T$-column sufficient tensor $\mathcal{M} = (m_{i_{1}i_{2}...i_{m}}) \in \mathbb{R}^{[m,n]},$ for any $k \in [n].$ Then, we have $m_{i_{1}i_{2}...i_{m}}^{(k)}=m_{i_{1}i_{2}...i_{m}},$ for all $i_{1},i_{2},...,i_{m} \in [k].$ Let $\mathcal{Z}  \in \mathbb{R}^{[m,k]}$ satisfies 
\begin{equation}\label{sub1}
z_{i_{1}i_{2}...i_{m}}(\mathcal{M}^{(k)}\mathcal{Z})_{i_{1}i_{2}...i_{m}} \leq {0},~\text{for all}~ i_{1},i_{2},...,i_{m} \in [k]. 
\end{equation}
We define $\mathcal{Y} = (y_{i_{1}i_{2}...i_{m}} ) \in \mathbb{R}^{[m,n]}$ such that
\begin{equation*}
   y_{i_{1}i_{2}...i_{m}}  = \begin{cases}
       z_{i_{1}i_{2}..i_{m}},~\text{if}~ i_{1},i_{2},...,i_{m} \in [k],\\
       0,~ \text{else}.
   \end{cases}
\end{equation*}
By Equation (\ref{sub1}), for all $i_{1},i_{2},...,i_{m} \in [n],$ we have 
$$y_{i_{1}i_{2}...i_{m}}(\mathcal{M}\mathcal{Y})_{i_{1}i_{2}...i_{m}} = z_{i_{1}i_{2}...i_{m}}(\mathcal{M}^{(k)}\mathcal{Z})_{i_{1}i_{2}...i_{m}} \leq 0.$$
Since $\mathcal{M}$ is $T$-column sufficient, we obtain
$$z_{i_{1}i_{2}...i_{m}}(\mathcal{M}^{(k)}\mathcal{Z})_{i_{1}i_{2}...i_{m}} = y_{i_{1}i_{2}...i_{m}}(\mathcal{M}\mathcal{Y})_{i_{1}i_{2}...i_{m}} = 0,~\text{for all}~i_{1},i_{2},...,i_{m} \in [k].$$
This implies that $\mathcal{M}^{(k)}$ is $T$-column sufficient for any $k \in [n].$
\end{proof}

A well-known result in LCP \cite{MR3396730} states that a matrix is a $P$-matrix (a matrix with positive principal minors)  if and only if it is column sufficient and non-degenerate (a matrix with nonzero principal minors). This naturally raises the question of whether an analogous characterization holds in the tensor space $\mathbb{R}^{[m,n]}$. In particular, one may ask whether a tensor $\mathcal{M} \in \mathbb{R}^{[2m,n]}$ is a $T$-$P$ tensor if and only if it is both $T$-column sufficient and $T$-non-degenerate. In what follows, we answer this question in the affirmative. We begin by introducing the notion of a $T$-non-degenerate tensor on the space $\mathbb{R}^{[m,n]}.$

\begin{definition}\rm\label{T-ND}
A tensor $\mathcal{M} \in \mathbb{R}^{[2m,n]}$ is said to be a $T$-non-degenerate tensor if and only if for any $\mathcal{Z} \in \mathbb{R}^{[m,n]} \setminus \{\mathcal{O}\}$, there exist $i_{1},i_{2},...,i_{m} \in [n]$ such that
$$z_{i_{1}i_{2}...i_{m}}(\mathcal{M}\mathcal{Z})_{i_{1}i_{2}...i_{m}} \neq 0.$$
\end{definition}

\begin{remark}\rm
\begin{itemize}
\item[\rm(i)] It is evident from the above definition that a $T$-non-degenerate tensor is a generalization of a non-degenerate matrix \citep[Definition 3.6.1]{MR3396730}.
\item[\rm(ii)] A $T$-$P$ tensor is a $T$-non-degenerate tensor, but the converse is not true in general (see Example \ref{T-ND but not T-P}).
\iffalse
\item[\rm(iii)] A  $T$-non-degenerate tensor is $T$-$R_{0}$, {\color{blue} but the converse statement need not be true (see Example \ref{T-R0 but not T-ND})}.
\fi
\end{itemize}
\end{remark}

\begin{example}\rm\label{T-ND but not T-P}
Let $\mathcal{M} = (m_{i_{1}i_{2}i_{3}i_{4}}) \in \mathbb{R}^{[4,2]},$ where $m_{1111} = m_{1212} = m_{2121}= m_{2222} =1, m_{1122}= m_{1221}=-1, m_{2112}= m_{2211}=-2$ and other entries be zero. We claim that \rm(i) $\mathcal{M}$ is $T$-non-degenerate,  \rm(ii) but not $T$-$P$. For any $\mathcal{Z} \in \mathbb{R}^{[2,2]}$, we have 
\begin{equation*}
\mathcal{M}\mathcal{Z} =   \begin{pmatrix}
z_{11}-z_{22} & z_{12}-z_{21}\\
z_{21}-2z_{12} & z_{22}-2z_{11}
\end{pmatrix}.
\end{equation*} 
\begin{itemize}
\item[\rm(i)] Let $\mathcal{Z} \in \mathbb{R}^{[2,2]} \setminus \{\mathcal{O}\}.$ Then we have the following cases:
\begin{itemize}
\item[\rm(a)] If exactly one of the entries of $\mathcal{Z}$ is nonzero, say, $z_{kl} \neq 0,$ for some $k,l \in [2]$ then $z_{kl}(\mathcal{M}\mathcal{Z})_{kl} = z_{kl}^{2} \neq 0.$
\item[\rm(ii)] If at least two of the entries of $\mathcal{Z}$ are nonzero, then due to the entries of $\mathcal{M}\mathcal{Z}$, we can see that there must exist at least one  $(i,j) \in [2] \times [2]$ such that $z_{ij}(\mathcal{M}\mathcal{Z})_{ij} \neq 0.$ 
\end{itemize}
From the above cases, we can conclude that $\mathcal{M}$ is a $T$-non-degenerate tensor.
\item[\rm(ii)] Let $\mathcal{Z} = \begin{pmatrix}
    1 & 1 \\
     1 & 1
 \end{pmatrix} \in \mathbb{R}^{[2,2]}.$ Then $\mathcal{M}\mathcal{Z} = \begin{pmatrix}
     0 & 0 \\
     -1 & -1
 \end{pmatrix}.$ Therefore, we obtain $z_{ij}(\mathcal{M}\mathcal{Z})_{ij} \\\leq 0$, for all $i,j \in [2].$ Thus $\mathcal{M}$ is not $T$-$P$. 
\end{itemize}
\end{example}

\iffalse
\begin{example}\rm\label{T-R0 but not T-ND}
  aa  
\end{example}
\fi

We now provide an equivalent condition for a $T$-$P$ tensor to be a $T$-column sufficient tensor.

\begin{theorem}
 Let $\mathcal{M} \in \mathbb{R}^{[2m,n]}.$  The following statements are equivalent.
 \begin{itemize}
     \item[\rm(i)]  $\mathcal{M}$ is a $T$-$P$ tensor.
     \item[\rm(ii)] $\mathcal{M}$ is $T$-column sufficient and $T$-non-degenerate.
 \end{itemize}
\end{theorem}
\begin{proof}
$\rm(i) \implies \rm(ii):$ Let $\mathcal{M}$ be a $T$-$P$ tensor. From the Definition \ref{T-ND} and item (ii) in Proposition  \ref{T-P is T-CS}, it implies that $\mathcal{M}$ is $T$-column sufficient and $T$-non-degenerate.\\
$\rm(ii) \implies \rm(i):$ Let $\mathcal{M}$ be $T$-column sufficient and $T$-non-degenerate. Assume contrary that $\mathcal{M}$ is not $T$-$P$. So there exists $\mathcal{Z} \in \mathbb{R}^{[m,n]} \setminus \{\mathcal{O}\}$ such that
$${z}_{i_{1}i_{2}...i_{m}}(\mathcal{M}\mathcal{Z})_{i_{1}i_{2}...i_{m}}
\leq {0},~{\forall}~ i_{1},i_{2},...,i_{m} \in [n].$$
Since $\mathcal{M}$ is $T$-column sufficient, the above equation yields 
\begin{equation}\label{eeqq}
{z}_{i_{1}i_{2}...i_{m}}(\mathcal{M}\mathcal{Z})_{i_{1}i_{2}...i_{m}}
= {0},~{\forall}~ i_{1},i_{2},...,i_{m} \in [n].
\end{equation}
Therefore there exists a nonzero tensor $\mathcal{Z} \in \mathbb{R}^{[m,n]}$ such that Equation (\ref{eeqq}) holds, which contradicts the fact that $\mathcal{M}$ is a $T$-non-degenerate tensor. Hence $\mathcal{M}$ must be a $T$-$P$ tensor.
\end{proof}

\section{Properties of the Solution Set of the \texorpdfstring{$\mathrm{TLCP}(\mathcal{M},\mathcal{Q})$}{TLCP(M,Q)}}\label{Properties}
In this section, we study properties of the solution set of the $\mathrm{TLCP}(\mathcal{M},\mathcal{Q})$, including convexity, uniqueness, and feasibility implying solvability, for $T$-column sufficient tensors.
\subsection{Convexity}\label{Convexity}
In the following, we prove that the convexity of  $\mathrm{SOL}(\mathcal{M},\mathcal{Q})$ is equivalent to $\mathcal{M}$ being a $T$-column sufficient tensor.

\begin{theorem}\label{T-CSimplyconvex}
Given $\mathcal{M} \in \mathbb{R}^{[2m,n]}$, the following statements are equivalent.
\begin{enumerate}
    \item[\rm(i)] $\mathcal{M}$ is $T$-column sufficient.
      \item[\rm(ii)] The set $\mathrm{SOL}(\mathcal{M},\mathcal{Q})$ is a convex set for all $\mathcal{Q} \in \mathbb{R}^{[m,n]}.$
\end{enumerate}
\end{theorem}
\begin{proof}
$\rm(i) \implies \rm(ii)$: Let $\mathcal{M} \in \mathbb{R}^{[2m,n]}$ be $T$-column sufficient. Let $\mathcal{Z}^{1}$ and $\mathcal{Z}^{2}$ be any two solutions of the $\mathrm{TLCP}(\mathcal{M},\mathcal{Q})$. This implies
\begin{equation}\label{Z^1}
\mathcal{Z}^{1} \geq \mathcal{O},~\mathcal{M}\mathcal{Z}^{1}+\mathcal{Q} \geq \mathcal{O}~\text{and}~\langle \mathcal{Z}^{1},\mathcal{M}\mathcal{Z}^{1}+\mathcal{Q}\rangle =0,~~\text{and}  
\end{equation}
\begin{equation}\label{Z^2}
\mathcal{Z}^{2} \geq \mathcal{O},~\mathcal{M}\mathcal{Z}^{2}+\mathcal{Q} \geq \mathcal{O}~\text{and}~\langle \mathcal{Z}^{2},\mathcal{M}\mathcal{Z}^{2}+\mathcal{Q}\rangle =0.
\end{equation}
To prove the convexity of $\mathrm{SOL}(\mathcal{M},\mathcal{Q})$, it suffices to show that for $\mathcal{Z}^{1}$  and $\mathcal{Z}^{2}$, Equation (\ref{cross commutative}) holds. Note that for each $i_{1},i_{2},...,i_{m}$ in $[n]$,  we have
\begin{eqnarray}\label{sol}
(\mathcal{Z}^{1}-\mathcal{Z}^{2})_{i_{1}i_{2}...i_{m}}(\mathcal{M}(\mathcal{Z}^{1}-\mathcal{Z}^{2}))_{i_{1}i_{2}...i_{m}} = {z}^{1}_{i_{1}i_{2}...i_{m}}(\mathcal{M}\mathcal{Z}^{1}+\mathcal{Q})_{i_{1}i_{2}...i_{m}}+ \notag\\
{z}^{2}_{i_{1}i_{2}...i_{m}}(\mathcal{M}\mathcal{Z}^{2}+\mathcal{Q})_{i_{1}i_{2}...i_{m}}-\notag\\
~~~~~~~~~~~~~~~~~~~\big[{z}^{1}_{i_{1}i_{2}...i_{m}}(\mathcal{M}\mathcal{Z}^{2}+\mathcal{Q})_{i_{1}i_{2}...i_{m}}+\notag\\{z}^{2}_{i_{1}i_{2}...i_{m}}(\mathcal{M}\mathcal{Z}^{1}+\mathcal{Q})_{i_{1}i_{2}...i_{m}}\big]
\end{eqnarray}
From Equations (\ref{Z^1}) and (\ref{Z^2}), we get $(\mathcal{Z}^{1}-\mathcal{Z}^{2})_{i_{1}i_{2}...i_{m}}(\mathcal{M}(\mathcal{Z}^{1}-\mathcal{Z}^{2}))_{i_{1}i_{2}...i_{m}} \leq {0},$ for all ${i_{1},i_{2},...,i_{m}} \in [n].$ Due to  $\mathcal{M}$ being $T$-column sufficient, this implies that $(\mathcal{Z}^{1}-\mathcal{Z}^{2})_{i_{1}i_{2}...i_{m}}(\mathcal{M}(\mathcal{Z}^{1}-\mathcal{Z}^{2}))_{i_{1}i_{2}...i_{m}} = {0},$ for all ${i_{1},i_{2},...,i_{m}} \in [n].$  Thus from Equation (\ref{sol}), we get \begin{equation}\label{nonnegative}
{z}^{1}_{i_{1}i_{2}...i_{m}}(\mathcal{M}\mathcal{Z}^{2}+\mathcal{Q})_{i_{1}i_{2}...i_{m}}+{z}^{2}_{i_{1}i_{2}...i_{m}}(\mathcal{M}\mathcal{Z}^{1}+\mathcal{Q})_{i_{1}i_{2}...i_{m}} =0,~\forall~ i_{1},i_{2},...i_{m} \in [n]. 
\end{equation}
Due to the nonnegativity of each term in Equation (\ref{nonnegative}), we have
\begin{equation*}
    \langle \mathcal{Z}^{1}, \mathcal{M}\mathcal{Z}^{2}+\mathcal{Q}\rangle = \langle \mathcal{Z}^{1}, \mathcal{M}\mathcal{Z}^{2}+\mathcal{Q}\rangle = 0.
\end{equation*}
From Proposition \ref{convexity}, it follows that $\mathrm{SOL}(\mathcal{M},\mathcal{Q})$ is a convex set for any $\mathcal{Q} \in \mathbb{R}^{[m,n]}.$\\
$\rm(ii) \implies \rm(i):$ Assume contrary that $\mathcal{M}$ is not $T$-column sufficient. This implies that there exists a tensor $\mathcal{Z} \in \mathbb{R}^{[m,n]}$ such that ${z}_{i_{1}i_{2},...i_{m}}(\mathcal{M}\mathcal{Z})_{i_{1}i_{2}...i_{m}} \leq {0}$ for all $i_{1},i_{2},...,i_{m}$ in $[n]$, with strict inequality holding for some indices, say $j_{1},j_{2},...,j_{m}$ in $[n]$. Let $\mathcal{X}^{1} = \mathcal{Z}^{+},$ $\mathcal{X}^{2}=\mathcal{Z}^{-}$, $\mathcal{Y}^{+} = \max\{\mathcal{M}\mathcal{Z},\mathcal{O}\},$ and $\mathcal{Y}^{-} = \max\{-\mathcal{M}\mathcal{Z},\mathcal{O}\}.$ Define 
\begin{equation*}
    \mathcal{Q} = \mathcal{Y}^{+}-\mathcal{M}\mathcal{Z}^{+} = \mathcal{Y}^{-}-\mathcal{M}\mathcal{Z}^{-}.
\end{equation*}
If ${z}_{i_{1}i_{2}...i_{m}}>0$, then $(\mathcal{M}\mathcal{Z})_{i_{1}i_{2}...i_{m}} \leq {0},$ and hence $(\mathcal{Y}^{+})_{i_{1}i_{2}...i_{m}} =0.$ Therefore, we have $(\mathcal{X}^{1})_{i_{1}i_{2}...i_{m}}(\mathcal{Y}^{+})_{i_{1}i_{2}...i_{m}} =0,$ which implies $\mathcal{X}^{1} \in \mathrm{SOL}(\mathcal{M},\mathcal{Q}).$ Similarly, we can show that if ${z}_{i_{1}i_{2}...i_{m}}<0$, then $\mathcal{X}^{2} \in \mathrm{SOL}(\mathcal{M},\mathcal{Q}).$ Thus $\mathcal{X}^{1}$ and $\mathcal{X}^{2}$ are two distinct solutions of the $\mathrm{TLCP}(\mathcal{M},\mathcal{Q}).$ Note that for the indices $j_{1},j_{2},...,j_{m} \in [n]$, we have ${z}_{j_{1}j_{2},...j_{m}}(\mathcal{M}\mathcal{Z})_{j_{1}j_{2}...j_{m}} < {0}$. Then, depending upon whether ${z}_{j_{1}j_{2},...j_{m}}>0$ or ${z}_{j_{1}j_{2},...j_{m}}$ $<0$, we can show that either $(\mathcal{X}^{1})_{j_{1}j_{2}...j_{m}}(\mathcal{Y}^{-})_{j_{1}j_{2}...j_{m}}>0$ or $(\mathcal{X}^{2})_{j_{1}j_{2}...j_{m}}(\mathcal{Y}^{+})_{j_{1}j_{2}...j_{m}}>0.$ This implies that either $\langle \mathcal{X}^{1}, \mathcal{M}\mathcal{Z}^{-}+\mathcal{Q}\rangle \neq 0$ or $\langle \mathcal{X}^{2}, \mathcal{M}\mathcal{Z}^{+}+\mathcal{Q}\rangle \neq 0,$ which is a contradiction to the convexity of $\mathrm{SOL}(\mathcal{M},\mathcal{Q})$. Hence $M$ is a $T$-column sufficient tensor.
\end{proof}

\begin{remark}\rm
If $m=1$, then the Theorem \ref{T-CSimplyconvex} reduces to \citep[Theorem 3.5.8]{MR3396730}.   
\end{remark}

\subsection{Feasibility implies Solvability}\label{FEASOL}
The following theorem presents a sufficient condition guaranteeing that the $\mathrm{TLCP}\\(\mathcal{M},\mathcal{Q})$ is solvable whenever it is feasible.

\begin{theorem}\label{feasimplysol}
 Let $\mathcal{M} \in \mathbb{R}^{[2m,n]}$ be a block symmetric and  $T$-column sufficient tensor. If the $\mathrm{TLCP}(\mathcal{M},\mathcal{Q})$ is feasible, then it is solvable.   
\end{theorem}
\begin{proof}
Given that $\mathcal{M} \in \mathbb{R}^{[2m,n]}$ is a block symmetric tensor and the $\mathrm{TLCP}(\mathcal{M},\mathcal{Q})$ is feasible. By Lemma (\ref{QPP1}), there exist tensors $\mathcal{Z}^{*}$ and $\mathcal{U}^{*}$ in $\mathbb{R}^{[m,n]}$ such that the conditions in Equations (\ref{KKT1})-(\ref{KKT5}) hold. Since $\mathcal{M}$ is $T$-column sufficient, from Equation (\ref{KKT5}), we get
\begin{equation}\label{QPP11}
   (\mathcal{Z}^{*}-\mathcal{U}^{*})_{i_{1}i_{2}...i_{m}}\big[\mathcal{M}(\mathcal{Z}^{*}-\mathcal{U}^{*})\big]_{i_{1}i_{2}...i_{m}} = {0},~\text{for all}~ i_{1},i_{2},...,i_{m} \in [n].  
\end{equation}
From Equations (\ref{KKT2}) and (\ref{KKT4}), we obtain 
\begin{eqnarray}
&{z}^{*}_{i_{1}i_{2}...i_{m}}\!(\mathcal{M}\mathcal{Z}^{*}\!\!+\!\mathcal{Q})_{i_{1}i_{2}...i_{m}}\!\!\!
+\! {z}^{*}_{i_{1}i_{2}...i_{m}}(\mathcal{M}(\mathcal{Z}^{*}\!\!-\!\mathcal{U}^{*}))_{i_{1}i_{2}...i_{m}} = {0},\!\!~{\forall}~ \!i_{1},i_{2},...,i_{m}\!\in\! [n] \label{QP11}\\ 
& \implies {z}^{*}_{i_{1}i_{2}...i_{m}}(\mathcal{M}(\mathcal{Z}^{*}-\mathcal{U}^{*}))_{i_{1}i_{2}...i_{m}} \leq {0},~{\forall}~ i_{1},i_{2},...,i_{m} \in [n].
\label{QPP12}
\end{eqnarray}
Also Equations (\ref{KKT1}) and (\ref{KKT3}) gives
\begin{eqnarray}\label{QPP13}
&{u}^{*}_{i_{1}i_{2}...i_{m}}\!(\mathcal{M}\mathcal{Z}^{*}\!+\!\mathcal{Q})_{i_{1}i_{2}...i_{m}}\!+\! {u}^{*}_{i_{1}i_{2}...i_{m}}(\mathcal{M}(\mathcal{Z}^{*}-\mathcal{U}^{*}))_{i_{1}i_{2}...i_{m}} \geq {0},~{\forall}~ i_{1},i_{2},...,i_{m} \in [n] \notag \\ 
& \implies -{u}^{*}_{i_{1}i_{2}...i_{m}}(\mathcal{M}(\mathcal{Z}^{*}-\mathcal{U}^{*}))_{i_{1}i_{2}...i_{m}} \leq {0},~{\forall}~ i_{1},i_{2},...,i_{m} \in [n]. 
\end{eqnarray}
Combining Equations (\ref{QPP11}), (\ref{QPP12}), and (\ref{QPP13}), we obtain 
\begin{equation*}
 {z}^{*}_{i_{1}i_{2}...i_{m}}(\mathcal{M}(\mathcal{Z}^{*}-\mathcal{U}^{*}))_{i_{1}i_{2}...i_{m}} = {0},~{\forall}~ i_{1},i_{2},...,i_{m} \in [n].   
\end{equation*}
Therefore from Equation (\ref{QP11}), we obtain ${z}^{*}_{i_{1}i_{2}...i_{m}}\!(\mathcal{M}\mathcal{Z}^{*} +\mathcal{Q})_{i_{1}i_{2}...i_{m}} = 0$, for all $i_{1},i_{2},...,i_{m} \in [n]$. Hence, our result follows.
\end{proof}

\subsection{Uniqueness}\label{Uniqueness} 
In this subsection, we provide a sufficient condition for the uniqueness of the solution to the $\mathrm{TLCP}(\mathcal{M},\mathcal{Q})$.

\begin{theorem}\label{T-CSimplyT-SP}
Let $\mathcal{M} \in \mathbb{R}^{[2m,n]}$ be a $T$-column sufficient tensor on $\mathbb{R}^{[m,n]}_{+}$. Then $\mathcal{M}$ is $T$-semi-positive.
\end{theorem}
\begin{proof}
Given that $\mathcal{M} \in \mathbb{R}^{[2m,n]}$ is $T$-column sufficient on $\mathbb{R}^{[m,n]}_{+}.$ This implies that for any $\mathcal{Z} \in \mathbb{R}^{[m,n]}_{+}$, we have
\begin{equation}\label{Proof1}
z_{i_{1}i_{2}...i_{m}}(\mathcal{M}\mathcal{Z})_{i_{1}i_{2}...i_{m}} \leq {0},~\forall~i_{j} \in [n] \implies z_{i_{1}i_{2}...i_{m}}(\mathcal{M}\mathcal{Z})_{i_{1}i_{2}...i_{m}} = {0},~\forall~i_{j} \in [n]. 
\end{equation}
Let $\mathcal{Z} \in \mathbb{R}^{[m,n]}_{+} \setminus \{\mathcal{O}\}$, and $I = \{(i_{1},i_{2},...,i_{m}) \in [n]^{m}: z_{i_{1}i_{2}...i_{m}} > {0}\}.$ We claim that there exists at least one $(i_{1},i_{2},...,i_{m}) \in I$ such that $(\mathcal{M}\mathcal{Z})_{i_{1}i_{2}...i_{m}} \geq 0.$ Assume contrary that for all $(i_{1},i_{2},...,i_{m}) \in I$, we have $(\mathcal{M}\mathcal{Z})_{i_{1}i_{2}...i_{m}} < 0.$ This gives
\begin{equation*}
z_{i_{1}i_{2}...i_{m}}(\mathcal{M}\mathcal{Z})_{i_{1}i_{2}...i_{m}}  \begin{cases}
    =0,~~\forall~(i_{1},i_{2},...,i_{m}) \notin I,\\<0,~~\forall~(i_{1},i_{2},...,i_{m}) \in I,
\end{cases}
\end{equation*}
which is a contradiction to Equation (\ref{Proof1}). Thus for any $\mathcal{Z} \in \mathbb{R}^{[m,n]}_{+} \setminus \{\mathcal{O}\}$, there exist indices $i_{1},i_{2},...,i_{m} \in [n]$  such that $z_{i_{1}i_{2}...i_{m}} > 0$ and $(\mathcal{M}\mathcal{Z})_{i_{1}i_{2}...i_{m}} \geq {0}.$ Hence $\mathcal{M}$ is $T$-semi-positive.
\end{proof}  

As a consequence of Theorems \ref{TSP} and \ref{T-CSimplyT-SP}, the following results are immediate.
\begin{theorem}\label{T-CCSimplyunique}
If $\mathcal{M} \in \mathbb{R}^{[2m,n]}$ is a $T$-column sufficient tensor on $\mathbb{R}^{[m,n]}_{+}$, then the $\mathrm{TLCP}(\mathcal{M},\mathcal{Q})$ has a unique solution for any $\mathcal{Q} \in \mathbb{R}^{[m,n]}_{++}.$
\end{theorem}

\begin{corollary}
If $\mathcal{M} \in \mathbb{R}^{[2m,n]}$ is a $T$-column sufficient tensor, then the $\mathrm{TLCP}$ $(\mathcal{M},\mathcal{Q})$ has a unique solution for any $\mathcal{Q} \in \mathbb{R}^{[m,n]}_{++}.$   
\end{corollary}

\begin{remark}\rm
 Theorem \ref{T-CCSimplyunique} differs from \citep[Theorem 14]{MR3778365} in the sense that it holds over $\mathbb{R}^{[m,n]}$. 
\end{remark}

\section{Conclusions}\label{conclusions}
In this paper, we introduced $T$-column sufficient tensors and studied their properties and relationships with other structured tensors, including $T$-positive semidefinite, $T$-$P$, and $T$-copositive tensors. We showed that the solution set of the $\mathrm{TLCP}(\mathcal{M},\mathcal{Q})$ is convex if and only if the tensor $\mathcal{M}$ is $T$-column sufficient. Furthermore, sufficient conditions were established for the uniqueness of solutions to the $\mathrm{TLCP}(\mathcal{M},\mathcal{Q})$ when $\mathcal{Q} > \mathcal{O}$, as well as for feasibility implying solvability via block symmetric tensors.

\section*{Acknowledgements}
Sonali Sharma  acknowledges the financial support from IIT Madras  under the institute post-doctoral fellowship (Award No.: F.Acad/IPDF/R12/2025). The third author's work is supported by the National Science and Technology Council, Taiwan (NSTC 114-2115-M-003-00-MY2; NSTC 114-2124-M-003-004) and the NTNU Higher Education Sprout Project - Center for Optimal Intelligent Data Analytics and Prediction.

\vspace{0.5cm}
\noindent{\bf{Data Availability}} This paper contains all the relevant data.
\section*{Declarations}
\textbf{Conflict of Interest} The authors have no competing interests to declare that are relevant to the content of this article.

\bibliographystyle{abbrv}
\bibliography{referencesall}

@article{LI2026116873,
  author    = {Xue-Liu Li and Guo-Ji Tang},
  title     = {Finiteness properties of the solution sets for horizontal tensor complementarity problems via structured tensor pair},
  journal   = {J. Comput. Appl. Math.},
FJOURNAL = {Journal of Computational and Applied Mathematics},
  volume    = {473},
  pages     = {116873},
  year      = {2026}
}

@article{li2025extended,
author = {Li, Xue-Liu and Jiang, Yi-Rong and Yuning Yang and Tang, Guo-Ji},
title = {Extended vertical tensor complementarity problems with finite solution sets},
journal = {J. Global Optim.},
fjournal={Journal of Global Optimization},
volume={92},
number={},
pages={431--452},
year = {2025}
}

@article{sharma2025extended,
  author    = {Sonali Sharma and V. Vetrivel},
  title     = {Extended horizontal tensor complementarity problems},
  journal   = {J. Optim. Theory Appl.},
 FJOURNAL = {Journal of Optimization Theory and Applications},
  volume    = {207},
  pages     = {44},
  year      = {2025}
}

@article{shang2023structured,
  author    = {Tong-Tong Shang and Guo-Ji Tang},
  title     = {Structured tensor tuples to polynomial complementarity problems},
  journal   = {J. Global Optim.},
 fjournal = {Journal of Global Optimization},
  volume    = {86},
  number    = {4},
  pages     = {867--883},
  year      = {2023}
}

@article {MR3989294,
    AUTHOR = {Huang, Zheng-Hai and Qi, Liqun},
     TITLE = {Tensor complementarity problems---{P}art {I}: {B}asic theory},
   JOURNAL = {J. Optim. Theory Appl.},
  FJOURNAL = {Journal of Optimization Theory and Applications},
    VOLUME = {183},
      YEAR = {2019},
    NUMBER = {1},
     PAGES = {1--23}
     }

@article {MR4023437,
    AUTHOR = {Huang, Zheng-Hai and Qi, Liqun},
     TITLE = {Tensor complementarity problems---{P}art {III}:
              {A}pplications},
   JOURNAL = {J. Optim. Theory Appl.},
  FJOURNAL = {Journal of Optimization Theory and Applications},
    VOLUME = {183},
      YEAR = {2019},
    NUMBER = {3},
     PAGES = {771--791}
     }

@article {MR3998357,
    AUTHOR = {Qi, Liqun and Huang, Zheng-Hai},
     TITLE = {Tensor complementarity problems---{P}art {II}: {S}olution
              methods},
   JOURNAL = {J. Optim. Theory Appl.},
  FJOURNAL = {Journal of Optimization Theory and Applications},
    VOLUME = {183},
      YEAR = {2019},
    NUMBER = {2},
     PAGES = {365--385}
     }

@article {MR3778365,
    AUTHOR = {Chen, Haibin and Qi, Liqun and Song, Yisheng},
     TITLE = {Column sufficient tensors and tensor complementarity problems},
   JOURNAL = {Front. Math. China},
  FJOURNAL = {Frontiers of Mathematics in China},
    VOLUME = {13},
      YEAR = {2018},
    NUMBER = {2},
     PAGES = {255--276}
     }

@article {MR3513266,
    AUTHOR = {Bai, Xue-Li and Huang, Zheng-Hai and Wang, Yong},
     TITLE = {Global uniqueness and solvability for tensor complementarity problems},
   JOURNAL = {J. Optim. Theory Appl.},
  FJOURNAL = {Journal of Optimization Theory and Applications},
    VOLUME = {170},
      YEAR = {2016},
    NUMBER = {1},
     PAGES = {72--84}
     }

@article {MR4310678,
    AUTHOR = {Palpandi, K. and Sharma, Sonali},
     TITLE = {Tensor complementarity problems with finite solution sets},
   JOURNAL = {J. Optim. Theory Appl.},
  FJOURNAL = {Journal of Optimization Theory and Applications},
    VOLUME = {190},
      YEAR = {2021},
    NUMBER = {3},
     PAGES = {951--965}
}

@book {MR3396730,
    AUTHOR = {Cottle, Richard W. and Pang, Jong-Shi and Stone, Richard E.},
     TITLE = {The Linear Complementarity Problem},
 PUBLISHER = {SIAM,
              Philadelphia, PA},
      YEAR = {2009}
     }

@article {MR3341670,
    AUTHOR = {Song, Yisheng and Qi, Liqun},
     TITLE = {Properties of some classes of structured tensors},
   JOURNAL = {J. Optim. Theory Appl.},
  FJOURNAL = {Journal of Optimization Theory and Applications},
    VOLUME = {165},
      YEAR = {2015},
    NUMBER = {3},
     PAGES = {854--873}
     }

@article {MR3501398,
    AUTHOR = {Song, Yisheng and Qi, Liqun},
     TITLE = {Tensor complementarity problem and semi-positive tensors},
   JOURNAL = {J. Optim. Theory Appl.},
  FJOURNAL = {Journal of Optimization Theory and Applications},
    VOLUME = {169},
      YEAR = {2016},
    NUMBER = {3},
     PAGES = {1069--1078}
     }

@article{li2024strict,
title={Strict feasibility for the polynomial complementarity problem},
author={Li, Xue-Liu and Tang, Guo-Ji},
 JOURNAL = {J. Global Optim.},
  FJOURNAL = {Journal of Global Optimization},
volume={89},
number={1},
pages={57--71},
year={2024}
}

@article{Li2020,
  title={Linear Complementarity Problem over Tensor Spaces},
  author={Li, Xia and Huang, Zheng-Hai },
  year={2020},
  journal   = {Sci. Sin. Math.},
FJOURNAL = {Scientia Sinica Mathematica},
  volume={50},
  number={9},
  pages={1169-1182}
  }

@article{Shang19062025,
author = {Shang, Tong-Tong and Jia, Wen-Sheng },
title = {The {GUS}-property of linear complementarity problems over tensor spaces},
journal = {Optimization},
volume = {0},
number = {0},
pages = {1--24},
year = {2025},
publisher = {Taylor \& Francis},
doi = {10.1080/02331934.2025.2517324},
}

@article{SHANG2024115383,
author = {Shang, Tong-Tong and Tang, Guo-Ji  and Jia, Wen-Sheng},
title = {Solution set bounds for {LCP}s over tensor spaces},
 journal   = {J. Comput. Appl. Math.},
FJOURNAL = {Journal of Computational and Applied Mathematics},
volume = {436},
pages = {115383},
year = {2024},
issn = {0377-0427},
doi = {https://doi.org/10.1016/j.cam.2023.115383},
}

\end{document}